\documentclass[11pt,leqno]{amsproc}
\usepackage{graphicx}
\usepackage{amsmath,amssymb,amsthm,mathtools}
\usepackage{mathrsfs}
\usepackage{enumitem}
\usepackage{microtype}
\usepackage{hyperref}
\usepackage{url}

\newtheorem{theorem}{Theorem}[section]
\newtheorem{proposition}[theorem]{Proposition}
\newtheorem{lemma}[theorem]{Lemma}
\newtheorem{corollary}[theorem]{Corollary}
\newtheorem{rem}[theorem]{Remark}
\newtheorem*{unremark}{Remark}
\theoremstyle{definition}
\newtheorem{definition}[theorem]{Definition}

\newcommand{\R}{\mathbb R}
\newcommand{\Max}{\operatorname{Max}}
\newcommand{\Sc}{\Sigma}
\newcommand{\cA}{\mathcal A}

\newcommand{\II}{\begin{itemize}}

\newcommand{\III}{\end{itemize}}

\begin{document}

\title[A bounded complete dcpo model of the cocountable real line]
{The set of real numbers with the cocountable topology\\ has a bounded complete dcpo model}

\author[D. Zhao]{Zhao Dongsheng}
\address[D. Zhao]{Mathematics and Mathematics Education, National Institute of
	Education, Nanyang Technological University, 1 Nanyang Walk, Singapore 637616}
\email{dongsheng.zhao@nie.edu.sg}

\author[C. Shen]{Shen Chong}
\address[C. Shen]{School of Science\\ Beijing University of Posts and Telecommunications\\ Beijing, PR China; Key Laboratory of Mathematics
	and Information Networks, Beijing University of Posts and Telecommunications, Beijing, China}
\email{shenchong0520@163.com}
\author[X. Xi]{Xi Xiaoyong}
\address[X. Xi]{School of Mathematics and Statistics\\
Jiangsu Normal University\\ Jiangsu, China} \email{littlebrook@jsnu.edu.cn}


\date{\today}

\begin{abstract}
We construct a bounded complete dcpo $P$ whose maximal point space is homeomorphic to the set $\R$ of all reals equipped with the cocountable topology, thereby answering an open problem. Using this dcpo, we obtain a new complete lattice whose Scott space is non-sober.
\end{abstract}

\maketitle

\medskip

\section{Introduction}
In domain theory, one tries to represent topological spaces as spaces of maximal points of posets, equipped with
the topology of interest. The most commonly used topologies on posets are the Scott and Lawson topologies.
A \emph{dcpo model} of a topological space $X$ is a dcpo $P$ for
which $\Max(P)$, with the subspace topology inherited from the Scott space $\Sc(P)$ of $P$, is
homeomorphic to $X$.

For each $T_1$ space $X$, Zhao and Xi constructed a dcpo model of $X$
\cite{XiZhao2018}.  Bounded completeness is substantially more restrictive:
Zhao and Xi proved that the positive integers with the cofinite topology do
not admit a bounded complete dcpo model, while many Hausdorff spaces (such as Hausdorff $k$-spaces) do
\cite{ZhaoXi2018}.  In the same paper they explicitly left open the case of
the real line with the cocountable topology. One motivation for this question is to determine whether every coherent and well-filtered space is homeomorphic to the maximal point space of a bounded complete dcpo (noting that the set $\R$ of all real numbers with the cocountable topology $\tau_{cc}$ is coherent and well-filtered).  The purpose of this paper is to
settle that case. We shall call $\R_{cc}=(\R, \tau_{cc})$ the cocountable real line.

\begin{theorem}[Main theorem]\label{thm:main}
The cocountable real line $\R_{\mathrm{cc}}$ has a bounded complete dcpo
model.
\end{theorem}

In the construction, we shall use  the  classical
set-theoretic result of Komj\'ath \cite{Komjath1984} (see also the exposition
of Soukup \cite{Soukup2010}): there is an almost disjoint family of countably
infinite subsets of a set of cardinality $\mathfrak c$ which refines every
uncountable subset.  We use the family to define a compact topology $\tau$ on $\R$.
We then use the $\tau$-closed sets of this compact topology to define the required bounded complete dcpo $P$, which is a model of $\R_{cc}$.

We shall also show that the Scott space of the constructed dcpo $P$ is non-sober and that adjoining a new top element and a new bottom element to $P$ yields a complete lattice whose Scott space is non-sober. Compared with the non-sober complete lattices in \cite{miao-2023,Isbell-1982}, our construction is more straightforward because it is built from subsets of $\R$.

\section{Preliminaries}\label{sec:prelim}

Throughout, $\omega=\{0,1,2,\ldots\}$, $\omega_1$ is the first uncountable
cardinal, and $\mathfrak c=|\R|=|2^\omega|$ denotes the cardinality of the
continuum.  For a set $M$ and a cardinal $\kappa$, we use the standard notation
\[
\begin{aligned}
[M]^\kappa&=\{A\subseteq M:|A|=\kappa\},\\
[M]^{\leq\kappa}&=\{A\subseteq M:|A|\leq\kappa\}.
\end{aligned}
\]
In particular,
\[
\begin{aligned}
[M]^\omega&=\{A\subseteq M:A\text{ is countably infinite}\},\\
[M]^{\leq\omega}&=\{A\subseteq M:A\text{ is finite or countably infinite}\},
\end{aligned}
\]
where the empty set and all finite sets belong to $[M]^{\leq\omega}$, and
\[
[M]^{<\omega}=\{A\subseteq M:A\text{ is finite}\}.
\]
Here ``countable'' means cardinality at most $\aleph_0$; thus a countable
set may be finite, while a member of $[M]^\omega$ is required to be
infinite.  We write $2=\{0,1\}$, $2^n$ for the set of binary sequences of
length $n$, and
\[
2^{<\omega}=\bigcup_{n<\omega}2^n,
\qquad
2^\omega=\{\xi:\omega\to 2\}.
\]
If $\xi\in 2^\omega$, then $\xi\mathord{\upharpoonright}n\in 2^n$ is its
restriction to the first $n$ coordinates.  We identify $\R$ with $2^\omega$
through a fixed bijection whenever a real number is used as a binary branch.

Let $(P,\leq)$ be a poset.  A nonempty subset $D\subseteq P$ is \emph{directed}
if for every $d_1,d_2\in D$ there is $d_3\in D$ with
$d_1\leq d_3$ and $d_2\leq d_3$.  A poset is a \emph{dcpo} if every directed
subset has a supremum.  It is \emph{bounded complete} if every nonempty
upper-bounded subset has a supremum.

For $A\subseteq P$, an element $u\in P$ is an upper bound of $A$ if
$a\leq u$ for all $a\in A$.  The supremum $\bigvee A$, when it exists, is
the least upper bound. A subset $U\subseteq P$ is an \emph{upper set} if
$x\in U$ and $x\leq y$ imply $y\in U$.

The \emph{Scott topology} $\Sc(P)$ on a poset $P$ consists of the upper sets
$U$ such that, whenever $D\subseteq P$ is directed and $\bigvee D\in U$,
then $D\cap U\neq\varnothing$.  The set of maximal elements of $P$ is denoted by
$\Max(P)$.

A \emph{dcpo model} of a topological space $X$ is a dcpo $P$ for
which $\Max(P)$, with the subspace topology inherited from $\Sc(P)$, is
homeomorphic to $X$.

For more information about the Scott and Lawson topologies, see
\cite{Gierz2003,GoubaultLarrecq2013}.

For an infinite set $M$, the \emph{cocountable topology} on $M$ is
\[
\tau_{\mathrm{cc}}(M)=\{\varnothing\}\cup
\{U\subseteq M:M\setminus U\text{ is countable}\}.
\]
We write $M_{\mathrm{cc}}$ for this space.

The space $M_{\mathrm{cc}}$ is $T_1$ and hyperconnected.  It is not compact
when $M$ is uncountable. In fact, the compact subsets of $M_{\mathrm{cc}}$ are exactly the finite subsets.

\begin{definition}
Let $M$ be a set.  A family $\mathcal F\subseteq[M]^\omega$ is
\emph{almost disjoint} (AD) if $A\cap B$ is finite whenever $A,B\in\mathcal F$
are distinct.  It is \emph{dense for uncountable subsets} if every
uncountable $B\subseteq M$ contains a member of $\mathcal F$.
\end{definition}
\begin{theorem}[Komj\'ath]\label{thm:komjath}
There is an AD family $\cA\subseteq[\R]^\omega$ of cardinality $\mathfrak c$
such that every uncountable $B\subseteq\R$ contains some $A\in\cA$.
\end{theorem}
\begin{rem}

Komj\'ath's formulation is often stated by saying that $\cA$ refines
$[\R]^{\omega_1}$.  In ZFC, every uncountable subset contains an
$\omega_1$-sized subset, so this is equivalent to the form used here.
The original reference is \cite{Komjath1984}; a detailed modern discussion
of dense families below the continuum is given in \cite{Soukup2010}.
\end{rem}

\section{A cross-almost-disjoint detector system}\label{sec:detectors}
We first strengthen Theorem~\ref{thm:komjath} by splitting every member into
continuum-many branches.
\begin{lemma}\label{lem:branch}
There is a family
\[
\{\mathcal S_p:p\in\R\},\qquad
\mathcal S_p\subseteq[\R\setminus\{p\}]^\omega,
\]
with the following properties:
\begin{enumerate}[label=(\roman*)]
\item any two distinct members of $\bigcup_{p\in\R}\mathcal S_p$ have finite
intersection;
\item for every $p\in\R$, every uncountable
$B\subseteq\R\setminus\{p\}$, and every countable
$C\subseteq\R\setminus\{p\}$, there is $S\in\mathcal S_p$ such that
$S\subseteq B\setminus C$.
\end{enumerate}
\end{lemma}
\begin{proof}
For each $A\in\cA$, both $A$ and $2^{<\omega}$ are countably infinite.
Choose an injection $e_A:2^{<\omega}\to A$.  We do not need the range of
$e_A$ to be all of $A$; the injection is used only to make different binary
branches meet at finitely many levels.  For $\xi\in2^\omega$ define the
branch
\[
A^\xi=\{e_A(\xi\mathord{\upharpoonright}n):n<\omega\}.
\]
Because $e_A$ is injective, every $A^\xi$ is countably infinite.  If
$\xi\neq\eta$, let $r$ be the first coordinate at which they differ.  For
every $n>r$ the finite sequences $\xi\mathord{\upharpoonright}n$ and
$\eta\mathord{\upharpoonright}n$ are distinct, hence their $e_A$-images are
distinct.  Therefore a common point can occur only among the first $r+1$
levels, and $A^\xi\cap A^\eta$ is finite.  Fix a bijection
$\beta:\R\to2^\omega$ and, for $p\in\R$, put
\[
\mathcal S_p=\{A^{\beta(p)}\setminus\{p\}:A\in\cA\}.
\]
Deleting $p$ removes at most one element, so every member of
$\mathcal S_p$ remains countably infinite and is disjoint from $\{p\}$.

Consider two supports
$S=A^{\beta(p)}\setminus\{p\}$ and
$T=A'^{\beta(q)}\setminus\{q\}$.  We first record that distinct parameter
pairs $(A,p)\neq(A',q)$ give distinct supports.  If $A\neq A'$, equality
$S=T$ would imply that the infinite set $S$ is contained in
$A\cap A'$, contradicting almost disjointness.  If $A=A'$ and $p\neq q$,
then equality would make the two branch sets have an infinite common subset,
contradicting the finite-intersection conclusion above.  Thus the supports
are genuinely different whenever their parameter pairs are different.

If $A\neq A'$, then $S\cap T\subseteq A\cap A'$, which is finite.  If
$A=A'$ but $p\neq q$, the branch calculation again gives $S\cap T$ finite.
This proves (i), including supports carrying different distinguished points.

Now let $B$ and $C$ be as in (ii).  The set $B\setminus C$ is uncountable.
By Theorem~\ref{thm:komjath}, choose $A\in\cA$ with
$A\subseteq B\setminus C$.  Every branch of $A$ is a subset of $A$, so in
particular $A^{\beta(p)}\setminus\{p\}\subseteq B\setminus C$.  Therefore this set
is the required member of $\mathcal S_p$.
\end{proof}
\section{The compact auxiliary topology}\label{sec:auxiliary}
For $p\in\R$, $S\in\mathcal S_p$, and finite $F\subseteq S$, define the
\emph{cone}
\[
K(p,S,F)=\{p\}\cup(S\setminus F).
\]
Let
\[
\mathscr K=\{\R\}\cup\{\{x\}:x\in\R\}\cup
\{K(p,S,F):p\in\R,\ S\in\mathcal S_p,\ F\in[S]^{<\omega}\}.
\]
We define $\tau$ to be the topology on $\R$ for which
\[
\{\R\setminus K:K\in\mathscr K\}
\]
is a subbase of open sets.
\begin{proposition}\label{prop:tau}
The space $(\R,\tau)$ is compact and $T_1$, and
\[
\tau\subseteq\tau_{\mathrm{cc}}.
\]
Equivalently, every proper $\tau$-closed subset of $\R$ is countable.
\end{proposition}
\begin{proof}
The complements of the singleton members of $\mathscr K$ are open, so every
singleton is $\tau$-closed and $(\R,\tau)$ is $T_1$.  Every member of
$\mathscr K$ other than $\R$ is countable: a cone is the union of one point
and a cofinite subset of a countably infinite set.  By the definition of the
topology, every $\tau$-closed set has the form
\[
F=\bigcap_{i\in I}\bigl(K_{i,1}\cup\cdots\cup K_{i,n_i}\bigr),
\qquad K_{i,j}\in\mathscr K.
\]
Here is the small point that is sometimes hidden in this notation.  The
closed subbasic sets are the members of $\mathscr K$, and arbitrary closed
sets are intersections of finite unions of these sets.  If $F\neq\R$, then at
least one finite union in the displayed representation is a proper subset of
$\R$; otherwise every factor would be $\R$ and their intersection would be
$\R$.  A proper finite union cannot contain the member $\R$ of $\mathscr K$,
so it is a finite union of singletons and cones and is therefore countable.
Consequently $F$ is a subset of a countable set.  This proves that every
proper $\tau$-closed set is countable.  Conversely, this is exactly the
statement that every nonempty $\tau$-open set has countable complement, i.e.
$\tau\subseteq\tau_{\mathrm{cc}}$.

To prove compactness, we use Alexander's subbase theorem in the
closed-subbase form. Here the complements of the members of
$\mathscr K$ form the chosen open subbase, so $\mathscr K$ is the associated
closed subbase.  A centered family with members containing a singleton $\{x\}$ already
has nonempty intersection, because every member of the family must meet
$\{x\}$ and hence must contain $x$.  The member $\mathbb R$, if present, can
be discarded without changing an intersection.  Suppose therefore that
$\mathcal F\subseteq\mathscr K$ consists only of cones and has empty
intersection.
Choose one cone $K_0\in\mathcal F$ and enumerate it without repetitions as
$K_0=\{x_0,x_1,\ldots\}$.  For each $n$, the total
intersection is empty, so choose $L_n\in\mathcal F$ with $x_n\notin L_n$.

We recursively choose distinct points $y_n$.  First, choose $y_0\in K_0\cap L_0$
and let $k_0$ be the unique index such that $y_0=x_{k_0}$.
For $n\geq 1$, having chosen
$y_0,\ldots,y_{n-1}$, for each $j<n$ let $k_j$ be the unique index satisfying
$x_{k_j}=y_j$.  Choose $r_n\ge n$ larger than all these finitely many
indices.  The finite subfamily
$\{K_0,L_0,\ldots,L_{r_n}\}$ is centered, so it has a point in its
intersection; choose
\[
y_n\in K_0\cap\bigcap_{i\leq r_n}L_i.
\]
This point is different from every earlier $y_j$: since $k_j\le r_n$, we
have $y_n\in L_{k_j}$, whereas the defining choice of $L_{k_j}$ gives
$y_j=x_{k_j}\notin L_{k_j}$.

For fixed $i,j$, every $y_n$ with $n\geq\max\{i,j\}$ lies in
$L_i\cap L_j$.  Hence $L_i\cap L_j$ is infinite.  Write
$L_i=K(p_i,S_i,F_i)$ and $L_j=K(p_j,S_j,F_j)$.  If the pairs
$(p_i,S_i)$ and $(p_j,S_j)$ were different, then $S_i$ and $S_j$ would be
distinct by the injectivity observation in the proof of
Lemma~\ref{lem:branch}.  Property (i) of that lemma would therefore make
$S_i\cap S_j$ finite.  Since
\[
L_i\cap L_j\subseteq (S_i\cap S_j)\cup\{p_i,p_j\},
\]
the intersection $L_i\cap L_j$ would be finite, contradicting the
infinitude proved above.  Hence $(p_i,S_i)=(p_j,S_j)$ for every $i,j$; there
is a single pair $(p,S)$ common to all members $L_n$.

For every $n$, the intersection $K_0\cap L_n$ contains all $y_m$ with
$m\ge n$, and is therefore infinite.  If the pair attached to $K_0$ were
different from $(p,S)$, cross-almost-disjointness would make
$K_0\cap L_n$ finite, a contradiction.  Thus $K_0$ also has pair $(p,S)$.
In particular $p\in K_0$.  Let $m$ satisfy $x_m=p$.  By construction,
$p=x_m\notin L_m$, whereas every cone with pair $(p,S)$ contains its apex
$p$.  This contradiction proves that every centered subfamily of
$\mathscr K$ has nonempty intersection, and Alexander's theorem now gives
compactness of $\tau$.
\end{proof}

\begin{proposition}[Fan property]\label{prop:fan}
For every $p\in\R$, every uncountable
$B\subseteq\R\setminus\{p\}$, and every countable
$C\subseteq\R\setminus\{p\}$, there is a decreasing sequence of nonempty
$\tau$-closed sets $(K_n)_{n<\omega}$ such that
\[
p\in K_n,\qquad K_n\cap C=\varnothing,\qquad
K_n\cap B\neq\varnothing,\qquad \bigcap_{n<\omega}K_n=\{p\}.
\]
\end{proposition}
\begin{proof}
Choose $S\in\mathcal S_p$ with $S\subseteq B\setminus C$ by
Lemma~\ref{lem:branch}.  Fix an enumeration $S=\{s_0,s_1,\ldots\}$ and
put $F_n=\{s_k:k<n\}$.  Then
\[
K_n=\{p\}\cup(S\setminus F_n)
\]
is one of the designated subbasic closed sets.  Since
$F_n\subseteq F_{n+1}$, the sequence is decreasing.  It contains $p$ and
meets $B$ because $S\subseteq B$; it misses $C$ because
$p\notin C$ and $S\cap C=\varnothing$.  Finally, every $s_k$ is removed
from $K_n$ once $n>k$, so the only point that belongs to all $K_n$ is $p$.
\end{proof}
\section{The pair-label bounded-complete dcpo}\label{sec:dcpo}
Let
\[
\mathcal C=\{C\subseteq\R:C\neq\varnothing
\text{ and }C\text{ is $\tau$-closed}\}.
\]
For $C\in\mathcal C$, write
\[
E_C=[\R\setminus C]^{\leq\omega}.
\]
Define
\[
P=\{(C,\Lambda):C\in\mathcal C,\ \Lambda\subseteq E_C\},
\]
with order
\[
(C,\Lambda)\leq(D,\Gamma)
\quad\Longleftrightarrow\quad
D\subseteq C\ \text{and}\ \Lambda\subseteq\Gamma.
\tag{5.1}
\]
The second coordinate is a set of countable subsets of $\R$; it is not itself
required to be countable.

\begin{theorem}\label{thm:bc}
The poset $P$ is a bounded complete dcpo.
\end{theorem}
\begin{proof}
First, (5.1) is a partial order.  Reflexivity is immediate.  If
$(C,\Lambda)\leq(D,\Gamma)$ and $(D,\Gamma)\leq(E,\Delta)$, then
$E\subseteq D\subseteq C$ and $\Lambda\subseteq\Gamma\subseteq\Delta$,
so transitivity holds.  If both inequalities hold, then $C=D$ and
$\Lambda=\Gamma$, proving antisymmetry.

The least element is $(\R,\varnothing)$: for every $(C,\Lambda)\in P$ we
have $(\R,\varnothing)\leq(C,\Lambda)$.  Let
\[
\mathcal B=\{(C_i,\Lambda_i):i\in I\}
\]
be a nonempty upper-bounded family, with upper bound $(D,\Gamma)$.  Then
$D\subseteq C_i$ for every $i$: this is exactly the first-coordinate
condition in $(C_i,\Lambda_i)\le(D,\Gamma)$.  Consequently, for every finite
collection $i_1,\ldots,i_k\in I$,
\[
D\subseteq C_{i_1}\cap\cdots\cap C_{i_k},
\]
so the closed sets $C_i$ have the finite-intersection property.  Compactness
of $(\R,\tau)$ gives
\[
C^*=\bigcap_{i\in I}C_i\neq\varnothing.
\]
Because an arbitrary intersection of closed sets is closed, $C^*$ belongs to
$\mathcal C$.
If $A\in\Lambda_i$, then $A\subseteq\R\setminus C_i$.  Since
$C^*\subseteq C_i$, we have $\R\setminus C_i\subseteq\R\setminus C^*$,
and hence $A\cap C^*=\varnothing$.  Thus every label occurring in any
$\Lambda_i$ is admissible for the limiting trace, and the union of the labels
is admissible:
\[
\Lambda^*=\bigcup_{i\in I}\Lambda_i\subseteq E_{C^*}.
\]
We claim that
\[
\bigvee\mathcal B=(C^*,\Lambda^*).
\]
It is an upper bound by (5.1): $C^*\subseteq C_i$ and
$\Lambda_i\subseteq\Lambda^*$ for every $i$.  If $(D',\Gamma')$ is another
upper bound, then $D'\subseteq C_i$ and $\Lambda_i\subseteq\Gamma'$ for
every $i$; hence $D'\subseteq C^*$ and $\Lambda^*\subseteq\Gamma'$, so
$(C^*,\Lambda^*)\leq(D',\Gamma')$.  Thus it is the least upper bound.

    Now let $D_0\subseteq P$ be directed.  We first check carefully that
the first coordinates have the finite-intersection property.  For one
element this is trivial.  Suppose that an element of $D_0$ lies above a
given finite list $(C_1,\Lambda_1),\ldots,(C_k,\Lambda_k)$.  Given one
more element $(C_{k+1},\Lambda_{k+1})$, directedness supplies an element
of $D_0$ above both this common upper bound and
$(C_{k+1},\Lambda_{k+1})$.  By transitivity it lies above all $k+1$
elements.  Induction therefore gives, for every finite subfamily of $D_0$,
an element $(C',\Lambda')\in D_0$ above that subfamily.  In particular,
$C'\subseteq C_j$ for every member of the finite subfamily, so the finite
intersection is nonempty.  Compactness and the label-union argument now
apply again.  Explicitly,
\[
\bigvee D_0=\left(\bigcap_{(C,\Lambda)\in D_0}C,
\bigcup_{(C,\Lambda)\in D_0}\Lambda\right).
\]
The intersection is nonempty by compactness and the union is admissible for
the same reason as above.  Hence $P$ is a dcpo and, by the first part,
bounded complete.
\end{proof}
\begin{unremark}
The upper-bound assumption is used exactly once: it supplies the nonempty
closed trace $D\subseteq C_i$ for every $i$.  Without an upper bound, the
intersection of the first coordinates could be empty, and the resulting pair
would not belong to $P$, since traces in $\mathcal C$ are required to be
nonempty.  Thus the proof establishes precisely bounded completeness.
\end{unremark}
\begin{proposition}\label{prop:max}
\[
\Max(P)=\{m_p:p\in\R\},\qquad
m_p=(\{p\},E_{\{p\}}).
\]
\end{proposition}
\begin{proof}
For any $(C,\Lambda)\in P$ choose $p\in C$.  Since
$\R\setminus C\subseteq\R\setminus\{p\}$, every member of $\Lambda$ is a
countable subset avoiding $p$; equivalently,
$\Lambda\subseteq E_C\subseteq E_{\{p\}}$.  Hence
$(C,\Lambda)\leq(\{p\},E_{\{p\}})=m_p$.  A maximal element must therefore
have a singleton first coordinate.  If that coordinate is $\{p\}$ but the
label is a proper subset of $E_{\{p\}}$, choose
$A\in E_{\{p\}}\setminus\Lambda$ and enlarge the label to
$\Lambda\cup\{A\}$; this is still an element of $P$ and is strictly above
$(\{p\},\Lambda)$.  Hence every maximal element is some $m_p$.  This
argument also shows why both coordinates must be saturated at a maximal
point: if the trace still contains two points, or if one admissible countable
label is missing, one can move strictly upward.

Conversely, suppose $m_p\leq(D,\Gamma)$.  The order forces the nonempty
closed set $D$ to satisfy $D\subseteq\{p\}$, hence $D=\{p\}$, and it also
forces $E_{\{p\}}\subseteq\Gamma$.  Since $(D,\Gamma)\in P$ gives the
opposite inclusion, $(D,\Gamma)=m_p$.
\end{proof}

\section{The maximal point space of $P$}\label{sec:scott}
\begin{lemma}\label{lem:markers}
For every countable $A\subseteq\R$, the set
\[
U_A=\{(C,\Lambda)\in P:A\in\Lambda\}
\]
is Scott-open, and
\[
U_A\cap\Max(P)=\{m_p:p\notin A\}.
\]
\end{lemma}
\begin{proof}
First, $U_A$ is an upper set: if $A\in\Lambda$ and
$(C,\Lambda)\leq(D,\Gamma)$, then $\Lambda\subseteq\Gamma$, so
$A\in\Gamma$.  Let $\mathcal D$ be directed and suppose its supremum lies
in $U_A$.  By the supremum formula, the label of the supremum is the union
of the labels in $\mathcal D$.  More explicitly, the directed-supremum
formula from Theorem~\ref{thm:bc} says that the second coordinate of
$\bigvee\mathcal D$ is exactly
$\bigcup_{(C,\Lambda)\in\mathcal D}\Lambda$.  Thus $A$ belongs to the label
of some member of $\mathcal D$, proving Scott inaccessibility.  Finally,
\[
A\in E_{\{p\}}\iff A\cap\{p\}=\varnothing\iff p\notin A,
\]
which proves the asserted maximal trace.
\end{proof}
\begin{theorem}\label{thm:topology}
Under the bijection $p\mapsto m_p$, the subspace Scott topology on
$\Max(P)$ is $\tau_{\mathrm{cc}}$.
\end{theorem}
\begin{proof}
Lemma~\ref{lem:markers} gives one inclusion immediately.  The empty set is
already relative Scott-open.  If $U\subseteq\R$ is nonempty and
cocountable, write $A=\R\setminus U$.  Then $A$ is countable and
$U=\{p:p\notin A\}$ is the trace of $U_A$ under the identification
$p\leftrightarrow m_p$.  Thus every cocountable open set is a relative
Scott-open set.

For the converse, let $O$ be Scott-open and put
\[
V=\{p\in\R:m_p\in O\}.
\]
We prove that either $V=\varnothing$ or $\R\setminus V$ is countable.
Suppose, towards a contradiction, that $p\in V$ and that
$B=\R\setminus V$ is uncountable.  Fix an arbitrary finite label
$\Lambda\subseteq E_{\{p\}}$ and put
\[
C_\Lambda=\bigcup\Lambda.
\]
Each member of $\Lambda$ is countable, so $C_\Lambda$ is countable; and
$p\notin C_\Lambda$ because every member of $\Lambda$ avoids $p$.  Hence
$B\setminus C_\Lambda$ is uncountable.  Apply Proposition~\ref{prop:fan}
with $B'=B\setminus C_\Lambda$ and $C=C_\Lambda$.  This gives a decreasing
closed fan $(K_n)$ with
\[
p\in K_n,\quad K_n\cap C_\Lambda=\varnothing,\quad
K_n\cap B'\neq\varnothing,\quad\bigcap_nK_n=\{p\}.
\]
For every $n$, the condition $K_n\cap C_\Lambda=\varnothing$ implies
$\Lambda\subseteq E_{K_n}$, so the following are indeed elements of $P$.
The family
\[
D_\Lambda=\{(K_n,\Lambda):n<\omega\}
\]
is directed: if $n\leq m$, then $K_m\subseteq K_n$, and hence
$(K_n,\Lambda)\leq(K_m,\Lambda)$.  By Theorem~\ref{thm:bc}, its supremum
is
\[
h_\Lambda=(\{p\},\Lambda).
\]
Indeed, the intersection of the first coordinates is $\{p\}$ by the fan
property, while the union of the constant label coordinates is just
$\Lambda$.  The pair is therefore the supremum by the explicit formula in
Theorem~\ref{thm:bc}, not merely an upper bound.
We claim that $h_\Lambda\notin O$.  If it belonged to $O$, Scott
inaccessibility would give $n$ with $(K_n,\Lambda)\in O$.  Choose
$b\in K_n\cap B'$.  Thus every $A\in\Lambda$ avoids $b$, and
therefore
\[
(K_n,\Lambda)\leq m_b
\]
because $\{b\}\subseteq K_n$ and $\Lambda\subseteq E_{\{b\}}$.  The
upper-set property of $O$ gives $m_b\in O$, contradicting $b\in B$.

Thus no finite-label history $h_\Lambda$ belongs to $O$.  On the other hand,
the family of all finite-label histories is directed: given finite labels
$\Lambda_1$ and $\Lambda_2$, their union is finite and
$h_{\Lambda_1\cup\Lambda_2}$ is a common upper bound.  Its supremum is
\[
\bigvee_{\Lambda\in[E_{\{p\}}]^{<\omega}}h_\Lambda
=m_p.
\]
To see the last equality directly, the first coordinates are all $\{p\}$,
and every countable set in $E_{\{p\}}$ occurs in one of the finite labels,
namely in the singleton family containing it.  Thus the union of all these
finite labels is exactly $E_{\{p\}}$.
Since $m_p\in O$, Scott inaccessibility gives a finite $\Lambda$ with
$h_\Lambda\in O$, a contradiction.  Therefore $\R\setminus V$ is countable
whenever $V$ is nonempty.  Hence the maximal Scott topology is exactly
$\tau_{\mathrm{cc}}$.
\end{proof}
\begin{corollary}[Explicit homeomorphism]
The map
\[
\Phi:\R_{\mathrm{cc}}\longrightarrow \Max(P),\qquad
\Phi(p)=m_p=(\{p\},E_{\{p\}}),
\]
is a homeomorphism.  Consequently $P$, together with the identification
$p\leftrightarrow m_p$, is a bounded-complete dcpo model of
$\R_{\mathrm{cc}}$.
\end{corollary}
\begin{proof}
Proposition~\ref{prop:max} says that $\Phi$ is bijective.  Theorem~\ref{thm:topology}
says that a subset $U\subseteq\R$ is cocountable-open exactly when
$\{m_p:p\in U\}$ is open in the Scott subspace on $\Max(P)$.  Thus both
$\Phi$ and $\Phi^{-1}$ are continuous.
\end{proof}

\begin{corollary}\label{cor:P-nonsober}
The Scott space of $P$ is not sober.
\end{corollary}
\begin{proof}
Let $O_1$ and $O_2$ be nonempty Scott-open subsets of $P$.  For
$i\in\{1,2\}$, choose $(C_i,\Lambda_i)\in O_i$ and $p_i\in C_i$.
As in the proof of Proposition~\ref{prop:max},
$(C_i,\Lambda_i)\leq m_{p_i}$, so $m_{p_i}\in O_i$.  Thus each
$O_i\cap\Max(P)$ is nonempty.  By Theorem~\ref{thm:topology}, these two
maximal-point traces correspond to nonempty cocountable subsets of $\R$,
and hence they intersect.  Therefore $O_1\cap O_2\neq\varnothing$, so the
whole Scott space $P$ is an irreducible closed set.

In a Scott space, the closure of a point $x$ is the principal ideal
$\mathord\downarrow x$.  Since $P$ has distinct maximal elements $m_p$, it
has no greatest element.  Hence $P$ is not the closure of any point and is
therefore not sober.
\end{proof}

\begin{theorem}\label{thm:nonsober-lattice}
Let $L=P\cup\{\top,\bot\}$ be obtained by adjoining new elements $\top$ and
$\bot$ such that $\bot<x<\top$ for every $x\in P$.  Then $L$ is a complete
lattice and its Scott space is not sober.
\end{theorem}
\begin{proof}
We first prove completeness.  Let $A\subseteq L$.  If
$A\subseteq\{\bot\}$, then $\bigvee A=\bot$.  If $\top\in A$, or if the
nonempty set $A\cap P$ has no upper bound in $P$, then $\bigvee A=\top$.
In every remaining case, $A\cap P$ is nonempty and upper-bounded in $P$, so
Theorem~\ref{thm:bc} gives
\[
\bigvee_L A=\bigvee_P(A\cap P).
\]
Thus every subset of $L$ has a supremum, and consequently $L$ is a complete
lattice.

Put $F=P\cup\{\bot\}$.  This is a lower set in $L$.  Let
$D\subseteq F$ be directed.  If $D\subseteq\{\bot\}$, then
$\bigvee D=\bot\in F$.  Otherwise, $D\cap P$ is a nonempty directed subset
of $P$, and its supremum in $P$ is also the supremum of $D$ in $L$.
Therefore $F$ is closed under directed suprema and is Scott-closed.

We next show that $F$ is irreducible.  Let $U$ and $V$ be Scott-open subsets
of $L$ that meet $F$.  If $\bot\in U$ or $\bot\in V$, then the corresponding
open set is all of $L$, and the conclusion is immediate.  Otherwise, both
$U$ and $V$ meet $P$.  The intersections $U\cap P$ and $V\cap P$ are
Scott-open in $P$ and, by the argument in Corollary~\ref{cor:P-nonsober},
their maximal-point traces are nonempty cocountable subsets of $\R$.
These traces intersect, so $U\cap V\cap F\neq\varnothing$.  Hence $F$ is
irreducible.

Finally, $F$ is not the closure of a point.  The closure of $\top$ is $L$,
whereas the closure of any $x\in F$ is $\mathord\downarrow x$, a proper
subset of $F$.  This is immediate for $x=\bot$; if $x\in P$, then the
principal ideal $\mathord\downarrow x$ contains at most one maximal element
of $P$, while $F$ contains all the distinct maximal elements $m_p$.  Thus
$F$ is an irreducible Scott-closed
set that is not the closure of any point, and the Scott space of $L$ is not
sober.
\end{proof}

\begin{rem}
Isbell constructed the first non-sober complete lattice
\cite{Isbell-1982}.  Miao et al. constructed a countable non-sober complete
distributive lattice \cite{miao-2023}.  Compared with these examples, our
construction is more straightforward because it is built from subsets of
$\R$.
\end{rem}


\begin{thebibliography}{99}
\small
\setlength{\itemsep}{-0.25em}
\bibitem{Engelking1989}
R.~Engelking,
\emph{General Topology},
Heldermann Verlag, Berlin, 1989.
\bibitem{Gierz2003}
G.~Gierz, K.~H.~Hofmann, K.~Keimel, J.~D.~Lawson, M.~Mislove and D.~S.~Scott,
\emph{Continuous Lattices and Domains},
Cambridge University Press, Cambridge, 2003.
\bibitem{GoubaultLarrecq2013}
J.~Goubault-Larrecq,
\emph{Non-Hausdorff Topology and Domain Theory: Selected Topics in Point-Set
Topology},
Cambridge University Press, Cambridge, 2013.

\bibitem{Isbell-1982}
J.~Isbell,
Completion of a construction of Johnstone,
\emph{Proc. Amer. Math. Soc.} \textbf{85} (1982), 333--334.

\bibitem{Komjath1984}
P.~Komj\'ath,
Dense systems of almost-disjoint sets,
in A.~Hajnal, L.~Lov\'asz and V.~T.~S\'os (eds.),
\emph{Finite and Infinite Sets}, Colloquia Mathematica Societatis
J\'anos Bolyai, vol.~37, North-Holland, 1984, pp.~527--536.
\bibitem{Lawson1997}
J.~D.~Lawson,
Spaces of maximal points,
\emph{Mathematical Structures in Computer Science} \textbf{7} (1997),
543--555.

\bibitem{miao-2023}
H.~Miao, X.~Xi, Q.~Li and D.~Zhao,
Not every countable complete distributive lattice is sober,
\emph{Mathematical Structures in Computer Science} \textbf{33} (2023),
no.~9, 809--831.
\url{https://doi.org/10.1017/S0960129523000269}
\bibitem{Soukup2010}
L.~Soukup,
Dense families of countable sets below $\mathfrak c$,
arXiv:1003.2496 (2010).
\url{https://arxiv.org/abs/1003.2496}.
\bibitem{XiZhao2017}
X.~Xi and D.~Zhao,
Well-filtered spaces and their dcpo models,
\emph{Mathematical Structures in Computer Science} \textbf{27} (2017),
507--515.
\url{https://doi.org/10.1017/S0960129515000171}.
\bibitem{XiZhao2018}
D.~Zhao and X.~Xi,
Directed complete poset models of $T_1$ spaces,
\emph{Mathematical Proceedings of the Cambridge Philosophical Society}
\textbf{164} (2018), 125--134.
\url{https://doi.org/10.1017/S0305004116000888}.
\bibitem{ZhaoXi2018}
D.~Zhao and X.~Xi,
On topological spaces that have a bounded complete dcpo model,
\emph{Rocky Mountain Journal of Mathematics} \textbf{48} (2018), no.~1,
141--156.
\url{https://doi.org/10.1216/RMJ-2018-48-1-141}.
\end{thebibliography}
\end{document}